\documentclass[12pt]{amsart}
\usepackage{a4}
\usepackage{amsthm, amsfonts, amssymb,latexsym}
\usepackage{enumerate, color}
\usepackage[english]{babel}
\usepackage[latin1]{inputenc}

\newtheorem{lemma}{Lemma}

\newtheorem{theorem}{Theorem}

\newtheorem{cor}{Corollary}

\newcommand{\F}{\mathbb{F}}

\newcommand\blfootnote[1]{%
  \begingroup
  \renewcommand\thefootnote{}\footnote{#1}%
  \addtocounter{footnote}{-1}%
  \endgroup
}
\newcommand{\beq}[1]{\begin{equation}\label{#1}}
\newcommand{\eeq}{\end{equation}}

\title[Elementary methods for Incidence problems in finite fields]{Elementary methods for Incidence problems in finite fields}

\author[J. Cilleruelo, A. Iosevich, B. Lund, O. Roche-Newton and M. Rudnev]{Javier Cilleruelo, Alex Iosevich, Ben Lund, Oliver Roche-Newton and Misha Rudnev}
\address{J. Cilleruelo: Instituto de Ciencias Matem\'aticas
 (CSIC-UAM-UC3M-UCM) and Departamento de Matem\'aticas, Universidad
 Aut\'onoma de Madrid, 28049 Madrid, Spain }
\email{franciscojavier.cilleruelo@uam.es}
\address{A. Iosevich: Department of Mathematics, University of Rochester, Rochester, NY }
\email{iosevich@gmail.com }
\address{B. Lund: Department of Computer Science, Rutgers, The State University of New Jersey, NJ }
\email{lund.ben@gmail.coms}
\address{O. Roche-Newton: Johann Radon Institute for Computational and Applied Mathematics (RICAM), Austrian Academy of Sciences, 4040 Linz, Austria }
\email{o.rochenewton@gmail.com }
\address{M. Rudnev: School of Mathematics, University Walk, Bristol, BS8 1TW }
\email{ M.Rudnev@bristol.ac.uk}

\begin{document}

\begin{abstract}
We use elementary methods to prove an incidence theorem for points and spheres in $\F_q^n$. As an application, we show that any point set $P\subset \F_q^2$ with $|P|\geq 5q$ determines a positive proportion of all circles. The latter result is an analogue of Beck's Theorem for circles which is optimal up to multiplicative constants.
\end{abstract}

\maketitle

\section{Introduction}
\blfootnote{Mathematics Subject Classification 52C10}
Let $\F_q$ be a field with characteristic strictly greater than 2.\footnote{Whenever a finite field $\F_q$ is mentioned in this paper, it is assumed to have characteristic strictly greater than 2.} In this note, it is established that for a point set $P \subset \mathbb{F}_q^d$ and a family $\mathcal S$ of \textit{spheres} in $\mathbb F_q^d$, the number of incidences between the points and spheres, which is denoted by $I(P,\mathcal S):=|\{(p,S)\in{ P \times \mathcal S}:p \in S\}|$, satisfies the bound

$$\frac{|P||\mathcal S|}{q} - |P|^{1/2}|\mathcal S|^{1/2}q^{d/2}<I(P,\mathcal S)< \frac{|P||\mathcal S|}{q} + |P|^{1/2}|\mathcal S|^{1/2}q^{d/2}.$$

Many results on incidence problems in finite fields have appeared in recent years; see for example \cite{BKT,HR,Jones, vinh}. For relatively large sets of points and surfaces  in $\F_q^d$, Fourier analysis and spectral graph theory have been the main tools to deal with these problems. For example, Vinh \cite{vinh} used the spectral method to prove that, for sets $P$ and $\mathcal L$ of points and lines respectively in $\mathbb{F}_q^2$,
\begin{equation}
I(P, \mathcal L)\leq \frac{|P||\mathcal L|}{q}+(|P||\mathcal L|q)^{1/2}.
\label{vinhST}
\end{equation}
The result was extended to incidences between points and hyperplanes in $\F_q^d$, and can also be proven using discrete Fourier analysis.

In \cite{Ci}, the first author found an elementary method to prove some results on combinatorial problems in finite fields, including an alternative proof of \eqref{vinhST}. Here we follow that elementary approach to give an estimate on incidences of points and spheres in $\F_q^d$ and illustrate it with some applications to the pinned distance problem, as well proving a version of Beck's Theorem for circles in $\mathbb{F}_q^2$.

In particular, the version of Beck's Theorem for points and circles is tight up to multiplicative constants. In the forthcoming Theorem \ref{beckcirc}, it is established that any set $P \subset \F_q^2$ such that $|P| \geq 5q$ determines\footnote{Just as in the real plane, three noncollinear points in $\F_q^2$ determine a unique circle. A set of points $P$ determines a given circles if three points from $P$ determine that circle. More details of these basic properties of circles are given in section 3.} a positive proportion of all circles. On the other hand, if one takes a set $P$ of $q$ points lying on a single line in the plane, then $P$ does not determine any circles. Similarly, if $P$ consists of, say, $q+1$ points on the same circle, then $P$ determines only $1$ circle. These degenerate examples are in a sense $1$-dimensional, and illustrate the tightness of Theorem \ref{beckcirc}.

Before saying any more about spheres in finite fields, it is necessary to define what is meant by such an object. We follow the notion of distance introduced in \cite{IR}; given a pair of points $x=(x_1,\cdots,x_d)$ and $y=(y_1,\cdots,y_d)$ in $\mathbb{F}_q^d$, the \textit{distance} between $x$ and $y$ is given by
$$||x-y||:=(x_1-y_1)^2+\cdots+(x_d-y_d)^2.$$
As one might expect, a \textit{sphere} in $\mathbb{F}_q^d$ is a set of points which are the same distance $\lambda$ from a given central point $(\alpha_1,\cdots,\alpha_d)$. That is, a sphere is the set of points $x=(x_1,\cdots,x_d)$ which satisfy an equation of the form
$$(x_1-\alpha_1)^2+\cdots+(x_d-\alpha_d)^2=\lambda.$$

These spheres have many natural properties which are analogous with spheres in $\mathbb{R}^d$. For example, given two circles\footnote{Naturally, we call a sphere in $\mathbb{F}_q^2$ a \textit{circle}.} in $\mathbb{F}_q^2$, it is easy to check that the circles intersect in at most two points. 

\subsection{Work of Phoung, Thang and Vinh}Shortly after an earlier draft of this paper was made available online, we became aware of its overlap with a forthcoming paper of Phoung, Thang and Vinh \cite{PTV}. The authors in \cite{PTV} give an independent proof of Theorem \ref{main} via graph theoretic methods similar to those used in \cite{vinh}. Further applications of the incidence result are also given in \cite{PTV}.

\section{Incidences between spheres and points}

Given finite sets $A,B$ in a finite group $(G,+)$ we use the notation \begin{equation*}r_{A+B}(x)=|\{(a,b)\in A\times B:\ a+b=x\}|.\end{equation*}
We recall the following elementary and well known identities:
\begin{eqnarray}
\label{1}\sum_{x\in G}r_{A+B}(x)&=&|A||B|,\\
\label{2}\sum_{x\in G}r^2_{A+B}(x)&=&\sum_{x\in G}r_{A-A}(x)r_{B-B}(x).
\end{eqnarray}
The quantity in \eqref{2} is called the \textit{additive energy} of $A$ and $B$.

\begin{lemma}Define
\begin{equation}A:=\{(a_1,a_2,\cdots,a_d,a_1^2+\cdots+a_d^2):\ a_1,\cdots,a_d\in \F_q\}\subset{\mathbb{F}_q^{d+1}}.
\label{Adefn}
\end{equation}
Then, for all  $x=(x_1,\cdots,x_{d+1})\ne (0,\cdots,0)$,
\begin{equation}\label{r}
r_{A-A}(x)\le q^{d-1}.
\end{equation}

\end{lemma}

\begin{proof}
This can be calculated directly. Indeed, the quantity $r_{A-A}(x)$ is the number of solutions
$$(a_1,\cdots,a_d,b_1,\cdots,b_d)\in{\mathbb{F}_q^{2d}}$$
to the system of equations
\begin{align*}
a_1-b_1 &= x_1
\\ a_2-b_2 &= x_2.
\\ &\vdots
\\ a_d-b_d &=x_d
\\ a_1^2+\cdots+a_d^2-b_1^2-\cdots-b_d^2 &=x_{d+1}.
\end{align*}
The $b_i$ variables can be eliminated, and this system of equations reduces to
\begin{equation}
2a_1x_1+\cdots+2a_dx_d-x_1^2-\cdots-x_d^2=x_{d+1}.
\label{hyperplane}
\end{equation}

If $x\ne 0$ then  there is some $1\leq{i}\leq{d}$ such that $x_i\neq{0}$. Otherwise $x_{d+1}=0$ and $x=0$ is the only choice which admits solutions to \eqref{hyperplane}. Without loss of generality, we may take $i=1$.
If we fix $a_2,\cdots,a_d$, then since the characteristic of the field is not equal to $2$, we have $2x_1\neq{0}$ and the value of $a_1$ is uniquely determined. This gives $r_{A-A}(x)\le q^{d-1}$.

If $x=0$, then trivially $r_{A-A}(0)=|A|=q^d$.

\end{proof}

\begin{lemma}\label{l2}Let $A$ be as defined in \eqref{Adefn},
and let $B,C\subset \mathbb{F}_q^{d+1}$ be arbitrary. Then
$$|\{(b,c)\in B\times C:\ b-c\in A\}|=\frac{|B||C|}q+\theta|B|^{1/2}|C|^{1/2}q^{d/2},$$
for some $\theta\in{\mathbb{R}}$ such that $|\theta|<1$.
\end{lemma}
\begin{proof}
Note that \begin{eqnarray*}
|\{(b,c)\in B\times C:\ b-c\in A\}|-\frac{|B||C|}q&=&\sum_{b\in B}\left (|\{c\in C:\ b-c\in A\}|-\frac{|C|}q\right )\\
&=&\sum_{b\in B}\left (r_{A+C}(b)-\frac{|C|}q\right ):=E.
\end{eqnarray*}
By Cauchy-Schwarz:
\begin{eqnarray*}|E|^2\le |B|\sum_{b\in B}\left (r_{A+C}(b)-\frac{|C|}q\right )^2
\le |B|\sum_{x\in \F_q^{d+1}}\left(r_{A+C}(x)-\frac{|C|}q\right)^2.
\end{eqnarray*}
Using \eqref{1}, \eqref{2} and the fact that $|A|=q^d$ we have
\begin{eqnarray*}\sum_{x\in \F_q^{d+1}}\left(r_{A+C}(x)-\frac{|C|}q\right)^2&=& \sum_{x\in \F_q^{d+1}}r_{A+C}^2(x)-q^{d-1}|C|^2\\
&=& \sum_{x\in \F_q^{d+1}}r_{A-A}(x)r_{C-C}(x)-q^{d-1}|C|^2\\
&\le & |A||C|+q^{d-1}\sum_{x\neq{0}}r_{C-C}(x)- q^{d-1}|C|^2\\
&=& |A||C|+q^{d-1}(|C|^2-|C|)- q^{d-1}|C|^2\\
&=& |C|q^{d-1}(q-1).\end{eqnarray*}
Thus, $|E|<(|B||C|)^{1/2} q^{d/2}$, which completes the proof.
\end{proof}

\begin{theorem} \label{main}Let $P\subset \F_q^d$ and let $\mathcal S$ be a family of spheres in $\F_q^d$. Then
$$\frac{|P||\mathcal S|}q-|P|^{1/2}|\mathcal S|^{1/2}q^{d/2}<I(P,\mathcal S)<\frac{|P||\mathcal S|}q+|P|^{1/2}|\mathcal S|^{1/2}q^{d/2}.$$
\end{theorem}
\begin{proof}
We denote by $S_{\alpha_1,\cdots,\alpha_d,\lambda}$ the sphere
$$\{(x_1,\cdots,x_d):\ (x_1-\alpha_1)^2+\cdots+(x_d-\alpha_d)^2=\lambda\}.$$
Define $$B=\{(p_1,\cdots,p_d,0):\ (p_1,\cdots,p_d)\in P\}$$ and $$C=\{(\alpha_1,\cdots,\alpha_d,-\lambda):\ S_{\alpha_1,\cdots,\alpha_d,\lambda}\in \mathcal S\}.$$
Note that $|B|=|P|$ and $|C|=|\mathcal S|$.

Now, note that
\begin{eqnarray*}
& & |\{(b,c)\in B\times C:\ b-c\in A\}|\\&=&|\{((p_1,\cdots,p_d,0),(\alpha_1,\cdots,\alpha_d,-\lambda))\in{B\times{C}}: (p_1-\alpha_1,\cdots,p_d-\alpha_d,\lambda)\in A\}|\\
&=&|\{((p_1,\cdots,p_d,0),(\alpha_1,\cdots,\alpha_d,-\lambda))\in{B\times{C}}:\ (p_1-\alpha_1)^2+\cdots+(p_d-\alpha_d)^2=\lambda)\}|\\
&=&I(P,\mathcal S).
\end{eqnarray*}
An application of Lemma \ref{l2} completes the proof.
\end{proof}

\section{Applications of the incidence bound}

\subsection{Pinned distances}

Let $P$ be a set of points in $\mathbb{F}_q^d$, and $y\in{\mathbb{F}_q^d}$. Following the notation of Chapman et. al. \cite{Ch}, let $\triangle_y(P)$ denote the set of distances between the point $y$ and the set $P$; that is
$$\triangle_y(P):=\{||x-y||:x\in{P}\}.$$
It was established in (\cite{Ch}, Theorem 2.3) that a sufficiently large set of points determines many pinned distances, for many different pins. Here, we use Theorem \ref{main} to give an alternative proof.

%
\begin{cor} \label{pinned} Let $P$ be a subset of $\mathbb{F}_q^d$ such that $|P|\geq{\epsilon^{-1}(1-\epsilon)^{1/2}q^{\frac{d+1}{2}}}$ for some $0<\epsilon<1$. Then,
\begin{equation}
\frac 1{|P|}\sum_{p\in P}|\triangle_p(P)|>(1-\epsilon)q.
\end{equation}

\end{cor}
\begin{proof} Fix a point $p\in{P}$, and construct a family of spheres $\mathcal S_p$ by minimally covering $P$ by concentric spheres around $p$. Note that $|\mathcal S_p|=|\triangle_p(P)|$, and that $I(P,\mathcal S_p)=|P|$. Repeating this process for each point in $P$, we generate a family of spheres $\mathcal S$ defined by the disjoint union
$$\mathcal S:=\bigcup_{p\in{P}}\mathcal S_p.$$
Observe that $I(P,\mathcal S)=\sum_{p\in{P}}I(P,\mathcal S_p)=|P|^2$. On the other hand, Theorem \ref{main} implies that
\begin{align*}|P|^2=I(P,\mathcal S)&<{\frac{|P||\mathcal S|}{q}+|P|^{1/2}|\mathcal S|^{1/2}q^{d/2}}
\\&=\frac{|P|\sum_{p\in{P}}|\triangle_p(P)|}{q}+|P|^{1/2}\left(\sum_{p\in{P}}|\triangle_p(P)|\right)^{1/2}q^{d/2}.
\end{align*}
If $\frac 1{|P|}\sum_{p\in P}|\triangle_p(P)|\le (1-\epsilon)q$ the last inequality would imply that $|P|<{\epsilon^{-1}(1-\epsilon)^{1/2}q^{\frac{d+1}{2}}}$.
\end{proof}

\begin{cor}\label{cor1}
Let $P$ be a subset of $\mathbb{F}_q^d$ such that $|P|\geq{\alpha^{-2}(1-\alpha^2)^{1/2}q^{\frac{d+1}{2}}}$ for some $0<\alpha<1$. Then,
\begin{equation*}
|\triangle_p(P)|> (1-\alpha)q
\end{equation*}
for at least $(1-\alpha)|P|$ points $p\in P$.
\end{cor}

\begin{proof}
Corollary \ref{pinned} implies that
$$\sum_{p\in P}|\triangle_p(P)|>(1-\alpha^2)q|P|.$$ On the other hand let
$$P'=\{p\in{P}:|\triangle_p(P)|\geq{(1-\alpha)q}\}$$ and suppose that $|P'|< (1-\alpha)|P|$. Then we would have
\begin{eqnarray*}\sum_{p\in P}|\triangle_p(P)|&=&\sum_{p\in{P\setminus{P'}}}|\triangle_p(P)|+\sum_{p\in P'}|\triangle_p(P)|\\ &<& (1-\alpha)q(|P|-|P'|)+q|P'|\\ &= & (1-\alpha)q|P|+\alpha q|P'|\\
&< & (1-\alpha^2)q|P|.\end{eqnarray*}
\end{proof}

\subsection{A version of Beck's Theorem for circles}

A result which is closely related to the Szemer\'{e}di-Trotter Theorem and incidence geometry is Beck's Theorem \cite{Beck}. This result states that a set of $N$ points in $\mathbb{R}^2$ determines $\Omega(N^2)$ distinct lines by connecting pairs of points, provided that the set of points does not contain a single line which supports $cN$ points, where $c$ is a small but fixed constant. We say that that $P$ determines a line $l$ if there exist two points $p_1$ and $p_2$ belonging to $P$ which both lie on the line $l$. In finite fields, an analogue of Beck's Theorem was proven by Alon \cite{Alon}, in the form of the following theorem:

\begin{theorem} \label{beck} Let $\epsilon >0$ and let $P$ be a set of points in the projective plane $\mathbb{P}\F_q^2$ such that $|P|>(1+\epsilon)(q+1)$. Then $P$ determines at least
$$\frac{\epsilon^2(1-\epsilon)}{2+2\epsilon}(q+1)^2$$
distinct straight lines.
\end{theorem}

Iosevich, Rudnev and Zhai \cite{IRZ} used Fourier analytic techniques to establish a similar result.
So, a sufficiently large set of points in the plane determines a positive proportion of all possible lines. The aim here is to establish an analogue of Theorem \ref{beck}, with the role of lines replaced by circles. Since there are $q^2$ choices for the location of a circle's centre, and $q$ choices for the radius, we want to generate $\Omega(q^3)$ circles. We first need an obvious definition of what it means for a circle to be generated by a set of points.

Given three non-collinear points in $\mathbb{R}^2$, there exists a unique circle which passes through each of the three points. The same is true of three points in $\mathbb{F}_q^2$:

\begin{lemma} \label{collinear} Let $(x_1,y_1)$, $(x_2,y_2)$ and $(x_3,y_3)$ be three distinct non-collinear points in $\mathbb{F}_q^2$. Then there exists a unique circle supporting the three points.
\end{lemma}

\begin{proof} We'll show that there exists a unique triple of elements $a,b,c\in{\mathbb{F}_q}$ such that the system of equations

\begin{equation}
\left\{ \begin{array}{llllllll}
                                          (x_1-a)^2 &+& (y_1-b)^2 &=&c,\\
                                          (x_2-a)^2 &+& (y_2-b)^2 &=&c,\\
                                          (x_3-a)^2 &+& (y_3-b)^2 &=&c,\\
                                        \end{array}\right.\label{system2}\end{equation}
can be realised. Note that we cannot have $x_1=x_2=x_3$, since this would contradict the hypothesis that the three points are non-collinear. Therefore, it is assumed without loss of generality that $x_1\neq{x_2}$ and $x_1\neq{x_3}$.

Subtracting the first equation in \eqref{system2} from the second and third yields the following system of linear equations:

\begin{equation}
\left\{ \begin{array}{llllllll}
                                          2(x_1-x_2)a &+& 2(y_1-y_2)b &=&x_1^2+y_1^2-x_2^2-y_2^2,\\
                                          2(x_1-x_3)a &+& 2(y_1-y_3)b &=&x_1^2+y_1^2-x_3^2-y_3^2.\\
                                        \end{array}\right.\label{system3}\end{equation}
It cannot be the case that both $y_2=y_1$ and $y_3=y_1$ (otherwise the three points would be collinear), and so at least one of the $b$ coefficients is non-zero. Therefore, this is a system of two linear equations with two variables ($a$ and $b$). This system has a unique solution $(a,b)$, unless it is degenerate. This solution can then be plugged into \eqref{system2} to give a unique solution to the system as required. It remains to show that \eqref{system3} is non-degenerate.

Suppose for a contradiction that \eqref{system3} is degenerate. Then there exists $\lambda\in{\mathbb{F}_q^*}$ such that

\begin{equation}
\left\{ \begin{array}{llllllll}
                                          \lambda{(x_2-x_1)} &=&x_3-x_1,\\
                                          \lambda{(y_2-y_1)} &=&y_3-y_1,\\
                                        \end{array}\right.\label{system4}\end{equation}
Since it is known that at least one of $y_2-y_1$ and $y_3-y_1$ is non-zero, and $\lambda$ is non-zero, it must be the case that both $y_2-y_1\neq{0}$ and $y_3-y_1\neq{0}$. Therefore,
$$\lambda=\frac{x_3-x_1}{x_2-x_1}=\frac{y_3-y_1}{y_2-y_1}.$$
Hence
$$y_3-y_1=(x_3-x_1)\frac{y_2-y_1}{x_2-x_1},$$
and clearly
$$y_2-y_1=(x_2-x_1)\frac{y_2-y_1}{x_2-x_1}.$$
This implies that $(x_1,y_1),(x_2,y_2)$ and $(x_3,y_3)$ are supported on a line with equation
$$y=\frac{y_2-y_1}{x_2-x_1}x+C,$$
where $C=\frac{y_1x_2-x_1y_2}{x_2-x_1}$. This is a contradiction, and the proof is complete.

\end{proof}

We say that a circle $C$ is determined by $P$ if there exist three points from $P$ which determine the circle $C$. We are now ready to state our version of Beck's Theorem for circles:

\begin{theorem} \label{beckcirc} Let $P\subset{\mathbb{F}_q^2}$ such that $|P|\geq{5q}$. Then $P$ determines at least $\frac{4q^3}{9}$ distinct circles.
\end{theorem}

Note that in the statement of Theorems \ref{beck} and \ref{beckcirc}, the conclusion is that we determine a positive proportion of all possible lines and circles respectively. If one asks how many points are needed to generate \textit{all} lines (respectively circles), then the problem becomes rather different, since one can take a point set $P=\F_q^2 \setminus l$ where $l$ is a line (respectively $P=\F_q^2 \setminus C$ where $C$ is a circle), and the line $l$ (respectively the circle $C$) is not determined by the point set $P$. So, we cannot hope to show that a set of $o(q^2)$ points determine all possible lines or circles.

\subsection{Proof of Theorem \ref{beckcirc}}

At the outset, identify a subset $P' \subset P$ such that $|P'|=5q$. The aim is to show that $P'$, and hence also $P$, determines many circles.

Let $\mathcal S$ be the set of all circles which contain less than or equal to $3$ points from $P'$. We will show that 

\begin{equation}
|\mathcal S|< \frac{5q^{3}}{9},
\label{aim}
\end{equation}
and then since there are $q^3$ circles, it must be the case that there are at least $\frac{4}{9}q^3$ circles which contain at least $3$ points from $P'$. These $4q^3/9$ circles are therefore spanned by $P$. 

It remains to prove \eqref{aim}, and to do this we will make use of the lower bound on $I(P',\mathcal S)$ from Theorem \ref{main}. We have

\begin{align*}
5|\mathcal S|-|P'|^{1/2}|\mathcal S|^{1/2}q&= \frac{|P'||\mathcal S|}{q}-|P'|^{1/2}|\mathcal S|^{1/2}q
\\&< I(P',\mathcal S)
\\& \leq 2|\mathcal S|,
\end{align*}
a rearrangement of which gives 
$$|\mathcal S| < \frac{|P'|q^2}{9}=\frac{5q^3}{9},$$
as required.
\begin{flushright}
\qedsymbol
\end{flushright}

Arguments similar to the proof above can be found in \cite{LS}. Note that, using a lower bound on the number of incidences between sets of points and lines in $\F_q^2$ which follows from the proof of the main result in \cite{vinh}, it is straightforward to adapt the proof of Theorem \ref{beckcirc} with lines in the place of circles in order to prove a version of Theorem \ref{beck}.

\subsection*{Acknowledgements}J.~Cilleruelo  was supported by grants MTM 2011-22851 of MICINN and ICMAT Severo
Ochoa project SEV-2011-0087. Oliver Roche-Newton was supported by EPSRC Doctoral Prize Scheme (Grant Ref:  EP/K503125/1) and by the Austrian Science Fund (FWF): Project F5511-N26, which is part of the Special Research Program ``Quasi-Monte Carlo Methods: Theory and Applications.

\end{document}